\documentclass{amsart}

\newtheorem{thm}{Theorem}
\newtheorem{prop}[thm]{Proposition}
\newtheorem{lem}[thm]{Lemma}

\newcommand{\set}[1]{\left\{{#1}\right\}}
\newcommand{\abs}[1]{\left\vert{#1}\right\vert}
\newcommand{\card}[1]{\left\vert{#1}\right\vert}
\newcommand{\C}{\mathbb C}
\newcommand{\R}{\mathbb R}

\newcommand{\Z}{\mathbb Z}

\newcommand{\res}[1]{|_{#1}}

\newcommand{\M}{\mathcal M}
\newcommand{\Fab}{F_{a,b}}
\newcommand{\Pab}{P_{a,b}}

\newcommand{\Int}{\mathop{\mathrm{Int}}\nolimits}
\newcommand{\Ind}{\mathop{\mathrm{Ind}}\nolimits}
\newcommand{\Res}{\mathop{\mathrm{Res}}\nolimits}

\begin{document}

\title[Principal series for Thompson's groups]{Analogs of principal series representations for Thompson's groups $F$ and $T$}

\author{{\L}ukasz Garncarek}
\address{University of Wroc{\l}aw, Institute of Mathematics, pl.~Grunwaldzki 2/4, 50-384 Wroc{\l}aw, Poland}
\email{Lukasz.Garncarek@math.uni.wroc.pl}


\begin{abstract}
We define series of representations of the Thompson's groups $F$ and $T$, which are analogs of principal series representations of $SL(2,\R)$. We show that they are irreducible and classify them up to unitary equivalence. We also prove that they are different from representations induced from finite-dimensional representations of stabilizers of points under natural actions of $F$ and $T$ on the unit interval and the unit circle, respectively.
\end{abstract}

\maketitle

\section{Introduction}

Suppose that a group $G$ acts on a measure space $(X,\mu)$, leaving $\mu$ quasi-invariant. Such an action gives rise to a series of unitary representations $\pi_s$ on the complex Hilbert space $L^2(X,\mu)$, given by
\begin{equation}\label{eq:rep_defn}
\pi_s(\gamma) f = \left(\frac{d\gamma_*\mu}{d\mu}\right)^{1/2+is}  f\circ\gamma^{-1}, 
\end{equation}
where $s\in\R$. If $G$ is a semisimple Lie group, a construction of this type can lead to a part of the principal series of $G$, and therefore to a large family of pairwise inequivalent irreducible representations (see e.g.~\cite{bekka2008kazhdan}, Appendix E).

If $G$ is a ``large'' group, such as the group of diffeomorphisms of a smooth manifold, the representations of the form~\eqref{eq:rep_defn} are irreducible (\cite{VerGelGra1982,Gar}). In this paper we study the irreducibility of representations of Thompson's groups $F$ and $T$ acting on the unit interval $I$ and unit circle $S^1$. Common wisdom says that they are ``large'' subgroups of the groups of piecewise linear homeomorphisms of $I$ and $S^1$. One might therefore expect that the corresponding representations are irreducible. In Section~\ref{sec:irrineq} we show that this is the case. Moreover, we show that the only equivalences between them arise from trivial reasons.

The main example we had in mind when considering the representations~\eqref{eq:rep_defn} was the principal series of $SL(2,\R)$. It consists of two parts, one of which is obtained from the action of $SL(2,\R)$ on the one-point compactification of $\R$ by M\"obius transformations. The same representations can be obtained through induction from a unitary character of the subgroup $P\leq SL(2,\R)$ consisting of upper-triangular matrices. This subgroup happens to be the stabilizer of the point $\infty$, hence it can be recovered from the action of $SL(2,\R)$ on $\R\cup\set{\infty}$.

It turns out that the analogy between the representations $\pi_s$ of Thompson's groups, and the principal series representations of $SL(2,\R)$ breaks here. In Section~\ref{sec:repind} we will show that there is a large class of pairwise inequivalent irreducible representations of $F$, obtained through induction from finite-dimensional unitary representations of stabilizers of points. This class, however, does not contain any of the representations $\pi_s$.

It is easy to see that if the restrictions of representations $\pi_s$ to a subgroup of $G$ are irreducible or inequivalent, then so are the original representations. It follows that the analogs of the principal series representations of many naturally occurring groups which contain $F$, such as the group of piecewise linear or piecewise projective homeomorphisms of the unit interval, are irreducible.

The author wishes to thank Jan Dymara and Tadeusz Januszkiewicz for helpful conversations and remarks.

\section{Preliminaries}

\subsection{Thompson's groups $F$ and $T$} \label{subsec:defFT}
The Thompson's group $F$ consists of all orientation-preserving piecewise linear homeomorphisms $\phi$ of the unit interval $[0,1]$ such that
\begin{enumerate}
\item $\phi$ preserves the set $\Z[1/2]\cap[0,1]$ of dyadic rational numbers in $[0,1]$,
\item $\phi$ has finitely many nondifferentiability points, which all lie in $\Z[1/2]$,
\item the slopes of $\phi$ are of the form $2^k$ for $k\in\Z$.
\end{enumerate}

The surjection $\eta\colon [0,1]\to S^1=\{z\in\C : \abs{z}=1\}$ given by $\eta(t)=e^{2\pi it}$ induces a faithful action of $F$ on $S^1$. Denote by $R$ the group of rotations of $S^1$ with angles in $\pi\Z[1/2]$. The Thompson's group $T$ is the group of homeomorphisms of $S^1$ generated by $F$, viewed as a group of homeomorphisms of $S^1$, and $R$. We will identify $F$ with the subgroup of $T$ fixing the point $1\in S^1$. 

\subsection{Technical lemmas}

\begin{lem}\label{lem:const}
Let $I\subseteq[0,1]$ be a nondegenerate closed interval of length less than $1$. If for a function $f\in L^1 ([0,1])$ there exists a dense subset $D$ of $S=\{s\in\R : s+I \subseteq [0,1]\}$ such that for every $s\in D$ we have $f(x)=f(x+s)$ almost everywhere on $I$, then $f$ is constant.
\end{lem}
\begin{proof}
For any continuous function $\phi\in C(I)$ we may define $f_\phi \colon S\to\C$ by
\begin{equation}
f_\phi(s) = \int_I f(x+s)\phi(x)\,dx = \int_{I+s} f(x)\phi(x-s)\,dx.
\end{equation}
Then $f_\phi$ is continuous, and constant on $D$, hence constant on $S$. For any subinterval $J\subseteq I$ we may take a uniformly bounded sequence $\phi_n\in C(I)$ converging pointwise to the characteristic function $\chi_J$, obtaining, by the dominated convergence theorem, that the value of the integral $\int_{J+s} f(x)\,dx$ is independent of $s\in S$. It is now easy to check that 
\begin{equation} \int_a^b f(x)\,dx = \frac{1}{b-a} \int_0^1 f(x)\,dx
\end{equation}
for any $0\leq a < b \leq 1$, hence $f$ is constant.
\end{proof}
\begin{lem}[\cite{CFP}, Lemma 4.2]\label{lem:F_exist}
If $0=x_0 < x_1 < x_2 < \cdots < x_n=1$ and $0=y_0 < y_1 < y_2 < \cdots < y_n=1$ are partitions of the interval $[0,1]$ consisting of dyadic rational numbers, then there exists $\gamma\in F$ such that $\gamma(x_i)=y_i$ for all $i=0,1,\ldots,n$. Furthermore, if $x_{i-1}=y_{i-1}$ and $x_i=y_i$ for some $i$, then $\gamma$ can be chosen so that $\gamma(x)=x$ for $x\in[x_{i-1},x_i]$.
\end{lem}
\begin{proof}
This follows from the observation that any two dyadic rational numbers can be written in the form $2^{\alpha_1}+2^{\alpha_2}+\cdots+2^{\alpha_k}$, where $\alpha_i\in\Z$, with the same number of summands. Such decompositions of $x_i-x_{i-1}$ and $y_i-y_{i-1}$ allow us to define $\gamma$ on $[x_{i-1},x_i]$.
\end{proof}
In particular, Lemma~\ref{lem:F_exist} implies that for any closed interval $I\subseteq(0,1)$, and $h\in\Z[1/2]$ such that $I+h\subseteq (0,1)$, there exists $\gamma_{I,h}\in F$ satisfying 
\begin{equation}
\gamma_{I,h}(x)=x+h
\end{equation}  
for $x\in I$. 

\section{A series of representations of the group $F$} \label{sec:irrineq}

The Lebesgue measure on $[0,1]$ is quasi-invariant under the action of $F$, which allows us to consider a series of representations $\pi_s$ of $F$, given by formula~\eqref{eq:rep_defn}. In this section we will show that they are irreducible and pairwise inequivalent, except for trivial cases. 

\subsection{Restrictions on intertwining operators}
For $0\leq a<b\leq 1$ we define $\Fab$ as the subgroup of $F$ consisting of homeo\-morphisms which are identity outside $[a,b]$, and $\Pab$ as the orthogonal projection onto the space $L^2([a,b])$, treated as a subspace of $L^2([0,1])$. Finally, for $\phi\in L^\infty([0,1])$, the corresponding multiplication operator on $L^2([0,1])$ is denoted by $M_\phi$, and the algebra of all such operators by $\M$. It is a maximal commutative operator algebra on $L^2([0,1])$ (e.g. \cite{Ped}, Proposition 4.7.6), hence $\M=\M'$. Moreover, $\M$ is equal to the von Neumann algebra generated by the operators $P_{a,b}$. 

\begin{lem}\label{lem:pres_Lab}
Fix $s\in\R$. If $f\in L^2([0,1])$ satisfies $\pi_s(\gamma)f=f$ for all $\gamma\in\Fab$, then $f\res{[a,b]} = 0$.
\end{lem}
\begin{proof}
If $f\res{[a,b]}\ne 0$, then there exist dyadic rational numbers $c<d$ such that $[c,d]\subset  (a,b)$ and $f\res{[c,d]}\ne 0$. By Lemma~\ref{lem:F_exist}, for a positive integer $n$ there exists $\gamma_n \in \Fab$ such that $\gamma_n(c)=c$ and $\gamma_n\res{[c,d]}$ is linear with slope $2^{-n}$. 

We have
\begin{equation}
\int_c^d \abs{\pi_s(\gamma_n^{-1})f}^2 = \int_c^{c+2^{-n}(d-c)} \abs{f}^2 \xrightarrow[n\to\infty]{} 0,
\end{equation}
and therefore $f \ne \pi_s(\gamma_n^{-1})f$ for some $n$.
\end{proof}

\begin{lem}\label{lem:int_diag}
For any $s,t\in \R$ the space $\Int(\pi_s,\pi_t)$ of intertwining operators is contained in $\M$.
\end{lem}
\begin{proof}
From Lemma~\ref{lem:pres_Lab} it follows that any $T\in \Int(\pi_s,\pi_t)$ commutes with the projections $\Pab$. We therefore have
\begin{equation}
    \M = \{\Pab : 0\leq a < b \leq 1\}'' \subseteq \Int(\pi_s,\pi_t)',
\end{equation}
so $\Int(\pi_s,\pi_t) \subseteq \Int(\pi_s,\pi_t)'' \subseteq \M' = \M$.
\end{proof}

\subsection{Irreducibility and non-equivalence}
\begin{prop}\label{prop:irreps}
The representations $\pi_s$ are irreducible for all $s\in\R$.
\end{prop}
\begin{proof}
Let $T\in\Int(\pi_s,\pi_s)$. It follows from Lemma~\ref{lem:int_diag} that there exists $\phi\in L^\infty([0,1])$ such that $T=M_\phi$. The condition $T\pi_s(\gamma)1=\pi_s(\gamma)T1$ implies
\begin{equation}
    \phi =\phi\circ\gamma^{-1}
\end{equation}
for every $\gamma\in F$. This is satisfied in particular by the elements $\gamma_{I,h}$, so by Lemma~\ref{lem:const} the function $\phi$ is constant and $T$ is a scalar operator.
\end{proof}

\begin{prop} \label{prop:equiv}
The representations $\pi_s$ and $\pi_t$ are unitarily equivalent if and only if $s-t=2k\pi/\log 2$ for some $k\in\Z$.
\end{prop}

\begin{proof}
If $s-t = 2k\pi/\log 2$ then $\left({d\gamma_*\mu}/{d\mu}\right)^{1/2+is} = \left({d\gamma_*\mu}/{d\mu}\right)^{1/2+it}$ and the representations are equivalent.

Now, suppose that there exists a unitary intertwining operator $U\in\Int(\pi_s,\pi_t)$. Then $U=M_\phi$, where $\phi\in L^\infty([0,1])$, and 
\begin{equation}\label{eq:uint}
    \frac{\phi\circ\gamma^{-1}}{\phi} = \left(\frac{d\gamma_*\mu}{d\mu}\right)^{i(s-t)}
\end{equation}
for every $\gamma\in F$. Putting $\gamma_{I,h}$ into~\eqref{eq:uint} we get $\phi(x)=\phi(x+h)$ almost everywhere on $I$. Therefore, by Lemma~\ref{lem:const} the function $\phi$ is constant. From~\eqref{eq:uint} it now follows that $(d\gamma_*\mu/d\mu)^{i(s-t)}$ is constant for all $\gamma\in F$, which is possible only if $s-t=2k\pi/\log 2$ for some $k\in\Z$.
\end{proof}

\section{Representations induced from stabilizers of points} \label{sec:repind}

\subsection{Irreducibility and nonequivalence criterion}
Let $H_1,H_2 \leq G$ be discrete groups, and let $\sigma_i$ be an irreducible finite-dimensional unitary representation of $H_i$. Denote $\rho_i=\Ind_{H_i}^G\sigma_i$. The following theorem is a slightly weaker version of Theorems 1--3 in~\cite{Obata1989}:

\begin{thm}\label{thm:obata}\ 
\begin{enumerate}
\item If $[H_1 : H_1 \cap gH_1g^{-1}] = \infty$ for every $g\in G - H_1$, then $\rho_1$ is irreducible. 
\item If $H_1=H_2$, and $[H_1 : H_1 \cap gH_1g^{-1}] = \infty$ for every $g\in G - H_1$, then $\rho_1$ and $\rho_2$ are inequivalent whenever $\sigma_1$ is not equivalent to $\sigma_2$,   
\item If $[H_1 : H_1 \cap gH_2g^{-1}] = \infty$ for every $g\in G$, then $\rho_1$ and $\rho_2$ are not equivalent.
\end{enumerate}
\end{thm}

\subsection{Representations induced from stabilizers of points}
For $p\in(0,1)$ let $F_p$ denote the stabilizer of the point $p$ under the action of $F$. Notice that if $p\not\in \Z[1/2]$, the map 
\begin{equation}
    \gamma \mapsto (\log_2 \gamma'_+(0), \log_2\gamma'(p), \log_2\gamma'_-(1))
\end{equation}
is a homomorphism from $F_p$ onto $\Z^3$ (for $p\in\Z[1/2]$ we just take two one-sided derivatives at $p$, obtaining a homomorphism onto $\Z^4$). This homomorphism allows us to construct a large family of characters of $F_p$.

\begin{lem}\label{lem:obata_asmp}
Let $p,q\in(0,1)$.
\begin{enumerate}
\item If $p$ and $q$ belong to different orbits of $F$, then $[F_p : F_p \cap \gamma F_q \gamma^{-1}] = \infty$ for every $\gamma\in F$.
\item $[F_p : F_p \cap \gamma F_p \gamma^{-1}] = \infty$ for every $\gamma\in F - F_p$.
\end{enumerate}
\end{lem} 
\begin{proof}
The group $F_p\cap F_q$ is the stabilizer of $q$ under the action of $F_p$, hence the index $[F_p : F_p\cap F_q]$ equals the cardinality of the orbit $F_pq$. It follows from Lemma~\ref{lem:F_exist} that $\card{F_pq} = \infty$ for any $q\not\in \{0,p,1\}$.

If $p$ and $q$ belong to different orbits of $F$, then $p\ne\gamma(q)$ for any $\gamma\in F$, and if $\gamma \in F-F_p$, then $p\ne\gamma(p)$, hence both conclusions follow.
\end{proof}

If $H$ is a subgroup of a discrete group $G$, and $\sigma$ is a unitary representation of $gHg^{-1}$, where $g\in G$, we define a unitary representation $\sigma^g$ of $H$ by $\sigma^g(x) = \sigma(gxg^{-1})$. It is irreducible if and only if $\sigma$ is irreducible. Moreover, it is not hard to check that the induced representations $\Ind_{gHg^{-1}}^G\sigma$ and $\Ind_H^G\sigma^g$ are equivalent.

\begin{prop}\label{prop:indreps_F}
Let $p,q\in (0,1)$, and let $\sigma_p$ and $\sigma_q$ be irreducible finite dimensional unitary representations of $F_p$ and $F_q$, respectively. Then the representations $\Ind_{F_p}^F \sigma_p$ and $\Ind_{F_q}^F \sigma_q$ are irreducible. Moreover, they are equivalent if and only if there exists $\gamma\in F$ such that $q = \gamma(p)$ and $\sigma_p$ is equivalent to $\sigma_q^\gamma$.
\end{prop}
\begin{proof}
Irreducibility and nonequivalence in case when $p$ and $q$ are in different orbits of $F$ follow directly from Theorem~\ref{thm:obata} and Lemma~\ref{lem:obata_asmp}. Now suppose that $q = \gamma(p)$ for some $\gamma\in F$. Then $\Ind_{F_q}^F \sigma_q$ is equivalent to $\Ind_{F_p}^F \sigma_q^\gamma$. But $\Ind_{F_p}^F \sigma_q^\gamma$ is equivalent to $\Ind_{F_p}^F \sigma_p$ if and only if $\sigma_q^\gamma$ and $\sigma_p$ are equivalent.
\end{proof}

\subsection{The representations $\pi_s$ are not induced from finite-dimensional representations of stabilizers of points}

\begin{prop}
Let $p\in(0,1)$ and let $\sigma$ be a finite-dimensional unitary representation of $F_p$. Then for all $s\in\R$ the representations $\Ind_{F_p}^F\sigma$ and $\pi_s$ are inequivalent.
\end{prop}
\begin{proof}
The restriction $\Res_{F_p}^F \Ind_{F_p}^F \sigma$ contains $\sigma$. Therefore it suffices to prove that the restriction $\Res_{F_p}^F\pi_s$ does not contain a finite-dimensional subrepresentation. 

The spaces $L^2([0,p])$ and $L^2([p,1])$ are $\Res_{F_p}^F\pi_s$-invariant. We will show that
$L^2([0,1]) = L^2([0,p]) \oplus L^2([p,1])$ is a decomposition of $\Res_{F_p}^F\pi_s$ into irreducible representations. Let $\rho=(\Res_{F_p}^F\pi_s)\res{L^2([0,p])}$, and take $T\in \Int(\rho,\rho)$. If $[a,b]\subseteq[0,p]$, then $F_{a,b}\leq F_p$, and by Lemma~\ref{lem:pres_Lab} the operator T commutes with the projection $P_{a,b}\res{L^2([0,p])}$. Hence, by the same argument as in the proof of Lemma~\ref{lem:int_diag}, T is a multiplication operator given by a function $\phi\in L^\infty([0,p])$. Proceeding as in the proof of Proposition~\ref{prop:irreps} we show that $\phi$ is constant. In the same way we show that $(\Res_{F_p}^F\pi_s)\res{L^2([p,1])}$ is irreducible. Both these representations are infinite-dimensional, so their sum can't contain a finite-dimensional representation.
\end{proof}

\section{Representations of the group $T$}

The pushforward of the Lebesgue measure on $[0,1]$ through the map $\eta$ from Subsection~\ref{subsec:defFT} is quasi-invariant under the action of $T$, hence we can define representations $\rho_s$ of $T$ through formula~\eqref{eq:rep_defn}. The map $\eta$ induces a unitary isomorphism of $L^2([0,1])$ and $L^2(S^1)$, which establishes equivalence of the representations $\pi_s$ and $\Res_F^T\rho_s$. We immediately obtain the following proposition:
\begin{prop}
The representations $\rho_s$ of\/ $T$ are irreducible. The representations $\rho_s$ and $\rho_t$ are unitarily equivalent if and only if $s-t=2k\pi/\log 2$ for some $k\in\Z$.
\end{prop}

For $p\in S^1$ denote by $T_p$ the stabilizer of $p$ under the action of $T$. By repeating the arguments of Section~\ref{sec:repind}, one can show the following analog of Proposition~\ref{prop:indreps_F} for the group $T$:
\begin{prop}\label{prop:indreps_T}
Let $p,q\in S^1$, and let $\sigma_p$ and $\sigma_q$ be irreducible finite dimensional unitary representations of $T_p$ and $T_q$, respectively. Then the representations $\Ind_{T_p}^T \sigma_p$ and $\Ind_{T_q}^T \sigma_q$ are irreducible. Moreover, they are equivalent if and only if there exists $\gamma\in T$ such that $q = \gamma(p)$ and $\sigma_p$ is equivalent to $\sigma_q^\gamma$.
\end{prop}

Besides the stabilizers of points of $S^1$, there exists one more natural subgroup of $T$, namely the group of rotations $R$ with angles in $\pi\Z[1/2]$. 

\begin{lem}\label{lem:obata_asmp_T}\
\begin{enumerate}
\item $[R : R \cap \gamma T_p \gamma^{-1}] = \infty$ for every $\gamma\in T$ and $p\in S^1$.
\item $[R : R \cap \gamma R \gamma^{-1}] = \infty$ for every $\gamma\in T - R$.
\end{enumerate}
\end{lem}
\begin{proof}
Since $\gamma T_p\gamma^{-1} = T_{\gamma(p)}$, and $R$ acts freely on $S^1$, we obtain (1). For the proof of (2) consider the action of T on $L^2(S^1)$ through the representation $\rho_0$. Then $R$ is the stabilizer $T_1$ of the function $1\in L^2(S^1)$, and $\gamma R \gamma^{-1}$ equals the stabilizer $T_f$ of the function $f = \rho_0(\gamma)1 = (d\gamma_*\mu/d\mu)^{1/2}$. For $\gamma\in T-R$ it is a non-constant simple function on $S^1$, hence its orbit under the action of $R$ is infinite, and so is the index $[R : R\cap \gamma R \gamma^{-1}]$.
\end{proof}

\begin{prop}
Take $p\in S^1$ and let $\rho$ and $\sigma$ be finite-dimensional unitary representations of $R$ and $T_p$, respectively. Then
\begin{enumerate}
\item the representation $\Ind_R^T\rho$ is irreducible,
\item the representations $\Ind_R^T\rho$ and $\Int_{T_p}^T\sigma$ are inequivalent,
\item for any $s\in\R$ the representations $\rho_s$ and $\Ind_R^T\rho$ are inequivalent.
\end{enumerate}
\end{prop}
\begin{proof}
The proofs of (1) and (2) follow from Theorem~\ref{thm:obata} and Lemma~\ref{lem:obata_asmp_T}. To prove (3) it suffices to show that $\Res_{T_p}^T\rho_s$ does not contain a finite-dimensional representation. But, treating $F$ as a subgroup of $T$ we have
\begin{equation}
    \Res_{F_p}^{T_p} \Res_{T_p}^T\rho_s = \Res_{F_p}^F\pi_s,
\end{equation}
and we already know that this restriction, and therefore also $\Res_{T_p}^T\rho_s$, do not contain a finite-dimensional representation.
\end{proof}

\bibliographystyle{amsplain}
\bibliography{thompson2}

\end{document}